\documentclass[a4paper]{article}
\usepackage{mathrsfs}
\usepackage{latexsym,bm}
\usepackage{amssymb,amsmath,amsthm}
\usepackage[english]{babel}
\usepackage{dsfont}
\usepackage{titlesec}
\usepackage[colorlinks=true]{hyperref}
\usepackage{graphicx}
\usepackage{color}
\usepackage{appendix}
\usepackage{todonotes}
\hypersetup{linkcolor=blue,urlcolor=red,citecolor=red}
\allowdisplaybreaks[4]

\newtheorem{theorem}{Theorem}[section]
\newtheorem{lemma}[theorem]{Lemma}
\newtheorem{remark}[theorem]{Remark}

\newtheorem{definition}[theorem]{Definition}

\numberwithin{equation}{section}

\title{Quantitative rapid stabilization for parabolic equations via the linear quadratic theory}
\author{Shengquan Xiang{\footnote{School of Mathematical Sciences, Peking University, Beijing 100871, China (e-mail: shengquan.xiang@math.pku.edu.cn).}}\and Yu Xiao{\footnote{School of Mathematics and Statistics, Wuhan University, Wuhan 430072, China (e-mail: xiaoyu\_math@whu.edu.cn).}}\and Can Zhang{\footnote{School of Mathematics and Statistics, Wuhan University, Wuhan 430072, China (e-mail: canzhang@whu.edu.cn).}}}
\date{}

\begin{document}
\selectlanguage{english}
\maketitle
\begin{abstract}
This paper addresses the problem of quantitative rapid stabilization via the linear quadratic (LQ) theory. Specifically, for non-self-adjoint parabolic equations, by virtue of the LQ theory, we derive the quantitative rapid stabilization by selecting a special cost functional and combining it with a quantitative observability inequality. Furthermore, this approach reveals the equivalence between the quantitative rapid stabilization and the quantitative observability inequality for linear systems. In addition, we apply this framework to the Navier--Stokes equations and establish the quantitative rapid stabilization around nontrivial steady states. Finally, we extend this methodology to finite-dimensional feedback laws in self-adjoint cases.
\end{abstract}


\section{Introduction}
Let $\Omega$ be a bounded, connected, open subset of $\mathbb{R}^d$ with smooth boundary, and let $\omega$ be a nonempty open subset of $\Omega$.  Consider the following controlled parabolic equation
\begin{equation}\label{EQ-FIRST}
\begin{cases}
\partial_t y-\Delta y+ay+b\cdot\nabla y=\chi_{\omega}u&\text{in } (0,+\infty)\times\Omega,\\
y=0&\text{on } (0,+\infty)\times \partial\Omega,\\
y(0,\cdot)=y_0\in L^2(\Omega),
\end{cases}
\end{equation}
where $a \in L^\infty(\Omega)$, $b \in L^\infty(\Omega)^d$ with $\text{div } b\in L^\infty(\Omega)$, the control $u$ belongs to $L^2(0,+\infty;L^2(\Omega))$, and $\chi_\omega$ denotes the characteristic function of $\omega$. The equation \eqref{EQ-FIRST} admits a unique solution $y\in C([0,+\infty);L^2(\Omega))\cap L^2_{\text{loc}}(0,+\infty;H^1_0(\Omega))$.

The controllability properties of \eqref{EQ-FIRST} have been extensively studied. Let $T\in (0,1)$. We recall below the \textit{observability inequality} (see, e.g., \cite[Theorem 3.5]{FZL})
\begin{equation}\label{OBS-INTRO}
\|\varphi(0,\cdot)\|_{L^2(\Omega)}\le e^{C\left(\Omega,d,\omega,a,b\right)/T} \|\varphi\|_{L^2((0,T)\times \omega)},
\end{equation}
for the adjoint equation
\begin{equation}\label{ADJOINT-INTRO-FIR}
\begin{cases}
-\partial_t \varphi - \Delta \varphi + a\varphi - \nabla \cdot (b\varphi)=0 & \text{in } (0,T)\times \Omega,\\
\varphi=0 & \text{on }(0,T)\times \partial\Omega,\\
\varphi(T,\cdot) \in L^2(\Omega).
\end{cases}
\end{equation}
By the classical duality principle (cf. \cite{TW} and \cite[Section 2.3]{Coron}), the equation \eqref{EQ-FIRST} is \textit{null controllable} at time $T$, i.e., there exists an open-loop control $u\in L^2(0,+\infty;L^2(\Omega))$\footnote{In fact, $u\in L^2(0,T;L^2(\Omega))$, and we extend $u$ by zero for all $t\ge T$.} such that the corresponding solution satisfies $y(T,\cdot)=0$, with the following control cost estimate:
\begin{equation}\label{C-COST-INTRO}
\|u\|_{ L^2(0,+\infty;L^2(\Omega))}\le e^{C\left(\Omega,d,\omega,a,b\right)/T}\|y_0\|_{L^2(\Omega)}.
\end{equation}

The stabilization problem for \eqref{EQ-FIRST} consists in designing a feedback law such that the energy dissipates asymptotically. Stabilization is closely related to controllability, but it also offers several additional advantages. First, unlike open-loop controllability, feedback stabilization provides a closed-loop mechanism that is more robust with respect to perturbations and uncertainties. Second, feedback laws are especially well suited for implementation in observer-based control frameworks; we refer to \cite{KF,LP1}. In many applications, the full state is not directly accessible, and the controller must rely on estimated states. In such situations, stabilizing feedback laws can be effectively combined with observers and finite-dimensional estimators.

Given a decay rate $\lambda>0$, one says that \eqref{EQ-FIRST} is exponentially stabilizable with rate $\lambda$ if there exists a feedback law such that the corresponding closed-loop system satisfies
$$\|y(t,\cdot)\|_{L^2(\Omega)}\le M_\lambda e^{-\lambda t}\|y_0\|_{L^2(\Omega)},\quad \forall t\ge 0,$$
for the stabilization cost $M_\lambda\ge 1$. If such a feedback law can be constructed for arbitrarily large $\lambda$, then the system is said to be \textit{rapidly stabilizable}. A variety of approaches have been developed to address the rapid stabilization for partial differential equations. Classical representative methods include the linear quadratic (LQ) theory \cite{B2,BW,LT}, spectrum-based methods \cite{BT2,CT1,KF,X2}, the Volterra backstepping \cite{BK,CN}, the Fredholm backstepping \cite{CL, GHMXZ,WaterWave}, weak observability inequalities \cite{LWXY,TWX}, the Gramian operator method \cite{Nyu,Urquiza}, among others. Each method possesses distinct advantages and a specific scope of applicability.

Among the available approaches, the LQ framework plays a particularly important role in stabilization theory. Starting from the seminal works of Lasiecka and Triggiani \cite{LT}, and subsequently those of Barbu, Wang, and many others \cite{B2,BR,BW}, LQ theory has provided a systematic way to construct stabilizing feedback laws through Riccati equations for broad classes of infinite-dimensional systems by solving suitable LQ optimal control problems. Moreover, LQ theory is closely connected with observability inequalities. In particular, it has been shown in \cite{TWX} that, for linear systems, stabilizability can be characterized in terms of appropriate weak observability inequalities. More recently, this characterization has led to several criteria for rapid stabilizability in \cite{KWY,LWXY}.

\subsection{Quantitative rapid stabilization problem}\label{subsec1.1}
While the classical rapid stabilization is qualitative in nature, a more refined question concerns the dependence of the stabilization cost on the prescribed decay rate. This leads to the notion of \textit{quantitative rapid stabilization}, which consists in constructing rapidly stabilizing feedback laws together with explicit estimates on the constant $M_\lambda$ in terms of $\lambda$. Such a quantitative description is of independent interest and has proved to be of considerable significance.

From a practical viewpoint, the quantitative rapid stabilization provides explicit information on the behavior of the closed-loop system. In particular, it makes it possible to estimate the time required for the energy to fall below a prescribed threshold, to compute half-life-type quantities, and to assess robustness with respect to perturbations. For example, under an additional bounded perturbation of size $\varepsilon$, standard perturbation arguments (cf. \cite{TW}) typically yield an estimate of the form
$$\|y(t,\cdot)\|_{L^2(\Omega)}\le M_\lambda e^{-(\lambda-\varepsilon M_\lambda)t}\|y_0\|_{L^2(\Omega)}, \quad\forall t\ge 0.$$
Hence, in order for the perturbed closed-loop system to remain exponentially stable, one needs at least $\varepsilon M_\lambda<\lambda$. This shows that the growth of $M_\lambda$ directly affects the admissible size of perturbations. Furthermore, if $M_\lambda=e^{C\lambda^\alpha}$ for $\lambda>1$, then, for $\alpha>1$, the decay term $e^{-\lambda t}$ cannot effectively compensate for the growth of $M_\lambda$ over short and intermediate times, which may result in a very large stabilization time. Consequently, it is particularly important to ensure that the stabilization cost grows at most like $e^{C\lambda}$. 

From the mathematical viewpoint, the quantitative rapid stabilization is closely connected with several other issues in control theory. On one hand, quantitative estimates for rapidly stabilizing feedback laws can be exploited to derive the small-time feedback null controllability with explicit control costs, as well as the finite-time stabilization through iterative or time-varying constructions. Relevant ideas were introduced in \cite{CN}, and further investigated in \cite{X0,X2}. On the other hand, this quantitative viewpoint can be combined with other problems; in particular, the frequency-Lyapunov method \cite{X2} has further inspired observer-based control, disturbance problems \cite{GPC}, and controllability for stochastic heat equations \cite{HLP}. In this sense, the quantitative rapid stabilization is not merely a strengthened form of the rapid stabilization, but rather a useful framework that bridges feedback stabilization, control cost estimates, and constructive controllability properties.

In recent years, several of the above methods used for rapid stabilization have been refined to yield quantitative results. For one-dimensional equations, substantial progress has been made. More precisely, backstepping techniques provide quantitative estimates by exploiting sharp bounds on the kernels of the associated Volterra transformations; see \cite{CN}. Moreover, Volterra backstepping has also been applied to the KdV equation \cite{X0}, while in \cite{Nyu} the Gramian operator method was used to study the finite-time stabilization for the KdV equation as well. We also refer to the recent work \cite{GHMXZ}, where a quantitative version of the Fredholm backstepping method was finally successfully established in the self-adjoint and skew-adjoint settings, and to the ongoing work \cite{XXZ}, where the modal decomposition method combined with the frequency-Lyapunov method yields quantitative results.\\[-2mm]

By contrast, for multidimensional parabolic equations, the quantitative rapid stabilization problem remains largely open. In the self-adjoint case, such as the heat equation, the frequency-Lyapunov method integrates a low-frequency feedback design with Lebeau--Robbiano--type spectral inequalities to derive explicit bounds for stabilization costs; see \cite{X1,X2}. However, the {\it non-self-adjoint case} has not yet been addressed. The main difficulty is that the spectral and orthogonality properties available in the self-adjoint setting are no longer present, which makes it much harder to derive quantitative estimates for the low-frequency dynamics. There is currently no quantitative rapid stabilization result for non-self-adjoint parabolic equations.\\[-2mm]

We begin by presenting several notations. Let $X$ be a Hilbert space. We use $\|\cdot\|_X$ and $\langle\cdot,\cdot\rangle_X$ to denote its norm and inner product, respectively. We write $X'$ for the dual space of $X$, and $\langle\cdot,\cdot\rangle_{X',X}$ for the duality pairing. For two Banach spaces $X_1,X_2$, we denote by $\mathcal{L}(X_1,X_2)$ the space of all bounded linear operators from $X_1$ to $X_2$ and set $\mathcal{L}(X_1):=\mathcal{L}(X_1,X_1)$. Let $0<T\le +\infty$. We then define the space
$$ W(0,T;X_1, X_2) := \left\{ f \in L^2(0,T; X_1) \,\bigg|\, \frac{d}{dt} f \in L^2(0,T; X_2) \right\},$$ 
equipped with the norm
$$\|f\|_{W(0,T; X_1, X_2)} := \left( \|f\|_{L^2(0,T; X_1)}^2 + \|\frac{d}{dt}f\|_{L^2(0,T; X_2)}^2 \right)^{\frac 12}.$$
Note that if $T\in(0,+\infty)$ and $X_1$ is continuously and densely embedded in $X_2$, then $W(0, T; X_1, X_2)$ is continuously embedded into $C\big([0, T]; [X_1, X_2]_{\frac{1}{2}}\big)$\footnote{Here, $[X_1,X_2]_{1/2}$ denotes the complex interpolation space of $X_1$ and $X_2$ with the interpolation exponent $\tfrac12$.}; see \cite{LM} for more details.

Let $Z$ be a function space consisting of functions $f=f(\cdot)$ defined on a subset of $\mathbb{R}$. For $\lambda>0$, we define the weighted space
$$Z_{\lambda} := \left\{ f \in Z \,\big|\, e^{\lambda t} f \in Z \right\} \;\textrm{ with the norm } \|f\|_{Z_{\lambda}} := \left( \|f\|_Z^2 + \big\| e^{\lambda t} f \big\|_Z^2 \right)^{\frac12}.$$

In what follows, for any time-space function $v=v(t,x)$, we often write $v(t)$ in place of $v(t,\cdot)$.  Furthermore, $C(\cdots)$, $C_1(\cdots)$, etc., stand for positive constants depending on the arguments in brackets.

\subsection{Quantitative rapid stabilization for multidimensional parabolic equations}\label{subsec1.2}

The first main result can be stated as follows.
\begin{theorem}\label{main-th-rapid-stab}
There exists a constant $\tilde{C}=\tilde{C}\left(\Omega,d,\omega,a,b\right)>0$ such that for any $\lambda>1$, there is a stationary feedback law $\mathcal{K}_{\lambda}\in\mathcal{L}(L^2(\Omega))$ so that the Cauchy problem \eqref{EQ-FIRST} with $u=\mathcal{K}_{\lambda}y$ admits a unique solution $y\in C([0,+\infty); L^2(\Omega))$ verifying
\begin{equation}\label{y-decay}
\|u(t)\|_{L^2(\Omega)}+\|y(t)\|_{L^2(\Omega)}\le e^{\tilde{C}\sqrt{\lambda}}e^{-{\lambda}t}\|y_0\|_{  L^2(\Omega)},\;\; \forall t>0.
\end{equation}
Moreover, the stabilizing feedback can be represented in the form $u=-\chi_\omega \mathcal P_\lambda y$, where $\mathcal P_\lambda\in\mathcal{L}(L^2(\Omega))$ is the solution of the operator Riccati equation 
\begin{equation}\label{PLAMBDA-61-RICCATI}
\left\langle (\Delta - a - b\cdot \nabla+\lambda I) {z},\, \mathcal P_\lambda  {z} \right\rangle_{L^2(\Omega)}-\frac12 \|\mathcal P_\lambda  {z}\|_{L^2(\omega)}^2+\frac12 \|\nabla  {z}\|_{L^2(\Omega)}^2=0,\quad \forall  {z}\in H^2(\Omega)\cap H_0^1(\Omega).\end{equation}
\end{theorem}

\begin{remark}\label{Remark-first-1}	
Several comments are given in order.
\begin{itemize}
\item[1.] This is the first result on the quantitative rapid stabilization for multidimensional non-self-adjoint parabolic equations. Note that the non-quantitative rapid stabilization of this model has been achieved in earlier works \cite{BT2, BW}.

\item[2.] Our method combines LQ theory directly with quantitative observability inequalities, and the analysis relies on two key ingredients. On one hand, we adopt a well-chosen cost functional
\begin{equation}\label{Func-alpha}
\mathcal J(y,u;y_0)=\frac12 \int_0^\infty e^{2\lambda t}\left(\|(-\Delta)^{\alpha} y(t)\|_{L^2(\Omega)}^2+\|u(t)\|_{L^2(\Omega)}^2\right)dt,\quad \alpha\in [0,1],
\end{equation}
and turn to the quantitative lower and upper bounds for the corresponding optimal control problem; {see Remark \ref{cost-func-remark}}. By choosing $\alpha=1/2$, we derive the lower bound with respect to the $L^2(\Omega)$ space presented in \eqref{LOWER-BOUND}. On the other hand, the observability inequality \eqref{OBS-INTRO} is used to obtain the quantitative upper bound stated in \eqref{PHI-bound-upp-523}, while the explicit blow-up rate of the cost is useful. Specifically, by applying the duality principle, we construct a  {feasible} control satisfying \eqref{C-COST-INTRO}, thereby establishing this bound.

The mechanism is different from that in most of the existing literature on the quantitative rapid stabilization. In many earlier works, quantitative estimates are obtained through explicit constructions based on spectral inequalities \cite{X1,X2}, kernel bounds \cite{CN,GHMXZ}, or other constructive techniques. Here, by contrast, the quantitative stabilization estimate is derived indirectly from a global observability inequality via the LQ framework. In this sense, the present approach is of a different nature from the more explicit constructions arising, for instance, in backstepping or frequency-Lyapunov methods. A related use of the Gramian method and control costs also appears in \cite{Nyu}.

\item[3.] In fact, when $a$ and $b$ also depend on time, the observability inequality \eqref{OBS-INTRO} still holds. Therefore, our result can be readily applied to the quantitative rapid stabilization of parabolic equations with time-dependent potential and convection terms. For the LQ framework for such nonautonomous equations, we refer the reader to \cite{BR}. The quantitative LQ framework developed in this paper can also be applied to boundary control problems, which will be addressed in our subsequent work. 
\end{itemize} 
\end{remark}

\subsection{An abstract equivalence between quantitative rapid stabilization and quantitative observability inequality}\label{ABS-SUBSET}
Our second main result is an {abstract} formulation of the above method. Let $X$ and $U$ be separable Hilbert spaces. Let $A:D(A)\subset X\to X$ be a strictly positive self-adjoint operator with a dense domain and compact resolvent. The perturbation operator $A_0$ belongs to $\mathcal{L}\big(D(A^{1/2}),X\big)$ and admits a continuous extension to $\mathcal{L}\big(X,D(A^{1/2})'\big)$. Let the control operator $B\in\mathcal L(U,X)$. We consider the linear control system
\begin{equation}\label{ABSTR-EQ-INTR}
\begin{cases}
\partial_t y+Ay+A_0y=Bu \quad \text{for } t>0,\\y(0)=y_0\in X,
\end{cases}
\end{equation}
where the control $u$ belongs to $L^2(0,+\infty;U)$. The equation \eqref{ABSTR-EQ-INTR} admits a unique solution $y\in C([0,+\infty);X)\cap L^2_{\text{loc}}(0,+\infty;D(A^{1/2}))$; see \cite{LM}.

\begin{definition}
 We say that the system \eqref{ABSTR-EQ-INTR} is {$e^{\mathcal O(\lambda^{\frac{p}{1+p}})}$-rapidly stabilizable} if there exists a constant $C>0$ such that, for every $\lambda>1$, one can find a feedback operator $K_\lambda\in \mathcal L(X,U)$, for which the solution of the closed-loop system with $u=K_\lambda y$ satisfies
\begin{equation}\label{q-rs-abstr}
\|u(t)\|_U+	\|y(t)\|_X\le e^{C\lambda^{\frac{p}{1+p}}}e^{-\lambda t}\|y_0\|_X,\quad \forall t>0.
\end{equation}
\end{definition}

\begin{definition}\label{Def-OBS-ABSTR}
 We say that the system \eqref{ABSTR-EQ-INTR} is {$e^{\mathcal O(\frac{1}{T^p})}$-observable} if there exists a constant $C>0$ such that, for every $T\in(0,1)$, the following observability estimate holds:
\begin{equation}\label{q-obs-abstr}
\|\phi(0)\|_X^2\le e^{\frac{C}{T^p}}\int_0^T \|B^*\phi(t)\|_U^2\,dt,
\end{equation}
for every solution $\phi$ to the adjoint system
$$\begin{cases}
\partial_t\phi-(A+A_0)^*\phi=0 \quad \text{for } t\in(0,T),\\
\phi(T)\in X.
\end{cases}$$
\end{definition}
Definition \ref{Def-OBS-ABSTR} introduces $e^{\frac{C}{T^p}}$-type observability inequalities, which quantify both the cost of recovering the initial state of the adjoint equation from small-time observations of $B^*\phi$ and, equivalently, the control cost of small-time null controllability for the original equation \eqref{ABSTR-EQ-INTR}. The blow-up rate of $e^{\frac{C}{T^p}}$ as $T \to 0^+$ has been extensively studied, starting with Seidman's work \cite{Seid} on the one-dimensional heat equation, which relies on solving  a ``window problem'' for series of complex exponentials. Subsequent developments include the moment method \cite{LP2,TT}, global Carleman estimates \cite{FZL,FI}, the spectral inequality method \cite{AEWZ,M}, the frequency-Lyapunov method \cite{X1,X2}, the frequency function method \cite{WK}, and propagation of smallness \cite{EMZ}. 

\begin{theorem}\label{EQV}
Let $p>0$. The control system \eqref{ABSTR-EQ-INTR} is $e^{\mathcal O(\lambda^{\frac{p}{1+p}})}\text{-rapidly stabilizable}$ if and only if it is $e^{\mathcal O(\frac{1}{T^p})}$-observable.
\end{theorem}

\begin{remark}
Based on the equivalence between weak observability inequalities and exponential stabilization introduced by Tr\'elat-Wang-Xu \cite{TWX}, the authors in \cite{LWXY} further established a family of weak observability inequalities to characterize the rapid stabilization. Notably, the dependence among the parameters involved in these inequalities has not yet been clarified. In this work, we propose quantitative observability inequalities as a strengthened version of those in \cite{LWXY}. Using the explicit form of the observation cost, we obtain a sharper result on the quantitative rapid stabilization. A possible further direction is to obtain quantitative rapid stabilization from a sequence of weak observability inequalities. 
\end{remark}

\subsection{Further application to Navier--Stokes equations around non-trivial states}
As an application to nonlinear equations, we establish the quantitative rapid stabilization for the two-dimensional Navier--Stokes equations around steady states. Define
$$\mathcal H:=\{y\in L^2(\Omega)^2:\ \operatorname{div}y=0\text{ in }\Omega,\ y\cdot n=0\text{ on }\partial\Omega\}\quad \text{with }\|y\|_{\mathcal H}:=\|y\|_{L^2(\Omega)^2} \text{ for }y\in \mathcal H,$$ 
$$\mathcal H_1:=\{y\in H^1_0(\Omega)^2:\ \operatorname{div}y=0\text{ in }\Omega\}\quad \text{with }\|y\|_{\mathcal H_1}:=\|\nabla y\|_{L^2(\Omega)^2}\text{ for }y\in \mathcal H_1.$$  
Let $(y_e,p_e)\in H^2(\Omega)^2\times H^1(\Omega)$ be a solution of the stationary system
$$\left\{\begin{aligned}
&- \Delta y_e + (y_e\cdot \nabla)y_e + \nabla p_e = f && \text{in } \Omega,\\
&\nabla \cdot y_e = 0 && \text{in } \Omega,\\
&y_e = 0 && \text{on } \partial\Omega.
\end{aligned}\right.$$ 

We consider the following controlled Navier--Stokes equation
\begin{equation}\label{NS-EQ-INTRO}
\left\{\begin{aligned}
&\partial_t y -  \Delta y + (y\cdot \nabla)y + \nabla p = f+\chi_\omega u && \text{in } (0,+\infty)\times \Omega,\\
&\nabla \cdot y = 0 && \text{in }  (0,+\infty)\times \Omega,\\
&y = 0 && \text{on }   (0,+\infty)\times \partial \Omega,\\
&y(0,\cdot)=y_0&& \text{in } \Omega.
\end{aligned}\right.
\end{equation} 
The theorem below reveals the quantitative rapid stabilization of \eqref{NS-EQ-INTRO} around the steady state $y_e$.

\begin{theorem}\label{main-th-rapid-stab-NS}
There exists a constant $\overline{C}=\overline{C}\left(\Omega,\omega,y_e\right)>0$ with the following property: for any $\lambda>1$, there exists a stationary feedback law $\mathcal{K}_{\lambda}\in\mathcal{L}(\mathcal H,L^2(\Omega)^2)$ such that, for any $y_0-y_e\in \mathcal H_1$ satisfying $\|y_0-y_e\|_{\mathcal H}\le e^{-\overline{C}\sqrt{\lambda}}$, the equation \eqref{NS-EQ-INTRO} with $u=\mathcal{K}_{\lambda}(y-y_e)$ admits a unique strong solution $(y,p)$ verifying
$y-y_e\in C([0,+\infty);\mathcal H_1)$ and
\begin{equation}\label{y-decay-2-NS}
\|u(t)\|_{L^2(\Omega)^2}+\|y(t)-y_e\|_{\mathcal H}\le e^{\overline{C}\sqrt{\lambda}}e^{-\lambda t}\|y_0-y_e\|_{  \mathcal H},\;\; \forall t>0.
\end{equation}
\end{theorem}

\begin{remark}\label{Remark-NS-1}
The proof of Theorem \ref{main-th-rapid-stab-NS} relies on the following quantitative observability inequality for the adjoint Oseen system associated with the linearized Navier--Stokes equations:
\begin{equation*}\label{NS-OBS}
\|\varphi(0,\cdot)\|_{\mathcal H} \le e^{C(\Omega,\omega,y_e)/T} \|\varphi\|_{L^2((0,T)\times \omega)^2}.
\end{equation*}
The quantitative observability inequality for Oseen systems is implied in \cite{FCGIP} by means of global Carleman estimates with cost $e^{C/T^4}$. More recently, the cost $e^{C/T^4}$ has been improved by Buffe and Takahashi \cite{BT} to $e^{C/T}$ by estimating the pressure term on the boundary differently. In \cite[Corollary 1.5]{BT} the result is stated for strong solutions; however, using standard arguments, their result can be easily extended to weak solutions.
\end{remark}

\begin{remark}\label{Remark-NS-2}
Hydrodynamic stability is an important topic, concerned with the stability of flows near a steady state or a Couette flow, and is of practical relevance. In particular, stabilization by means of finite-dimensional feedback laws is especially meaningful, both from the theoretical and practical perspectives; relevant results on finite-dimensional rapidly stabilizing feedback laws for the Navier--Stokes equations can be found in \cite{BT2,BR,X1}.
\end{remark}

\subsection{A finite-dimensional feedback extension for the  heat equation and an open problem}
As demonstrated in Theorem \ref{main-th-rapid-stab}, the feedback control law is constructed via an infinite-dimensional operator Riccati equation. In fact, for the heat equation, our quantitative LQ framework is also applicable to finite-dimensional controls expressed as follows, and it recovers the quantitative result by Xiang \cite{X2},
$$u(t,x)=\sum_{1\leq j\leq N} u_j(t)\varphi_j(x).$$
The precise statement is given in Section~\ref{QRS-FINITE}. Finite-dimensional feedback controls are particularly relevant in practice, as they involve only finitely many control channels and are thus easier to implement.

Nevertheless, this finite-dimensional extension relies heavily on the self-adjoint structure of the Laplacian. This method can also be applied to the finite-dimensional quantitative rapid stabilization for the Navier--Stokes equations around the zero steady state. Whether a finite-dimensional feedback law can be established for the non-self-adjoint parabolic equation \eqref{EQ-FIRST} still remains unclear. In particular, whether the frequency-Lyapunov approach \cite{X2} that relies on spectral estimates can be used for non-self-adjoint cases is still largely open. 
\vspace{2mm}

\noindent \textbf{Organization of the paper.}
The rest of the paper is organized as follows. Section \ref{QRS-INF} establishes the quantitative rapid stabilization for non-self-adjoint parabolic equations, thereby proving Theorem \ref{main-th-rapid-stab}. Section \ref{EQUV-SEC} is devoted to proving the equivalence between quantitative rapid stabilization and quantitative observability,  i.e., Theorem \ref{EQV}. Section \ref{NS-QRS-SEC} achieves the quantitative rapid stabilization of \eqref{NS-EQ-INTRO} around steady states, corresponding to Theorem \ref{main-th-rapid-stab-NS}. Section \ref{QRS-FINITE} constructs a finite-dimensional feedback law for the heat equation.

\section{Quantitative rapid stabilization for non-self-adjoint parabolic equations}\label{QRS-INF}

This section is devoted to proving Theorem \ref{main-th-rapid-stab} and illustrating our quantitative LQ framework. To this end, we formulate the following LQ optimal control problem 
\begin{equation}\label{Co-FUNC-INTRO}
\Phi_1(y_0):= {\inf}\left\{\mathcal J_1(y,u;y_0)\;\middle|\; 
\begin{array}{l}
(y,u)\in W_{{\lambda }}(0,+\infty;H_0^1(\Omega),H^{-1}(\Omega))\times L^2_{\lambda}(0,+\infty;L^2(\Omega)),\\[1mm]
(y,u)\ \text{satisfies \eqref{EQ-FIRST}}
\end{array}\right\},
\end{equation}
where the cost functional is
\begin{equation}\label{cost-infiite-intro}
\mathcal J_1(y,u;y_0)=\frac12\int_0^\infty e^{2\lambda t}\left(\|\nabla y(t)\|_{L^2(\Omega)}^2+\|u(t)\|_{L^2(\Omega)}^2\right)\,dt.
\end{equation}
The weighted spaces $W_{\lambda}$ and $L^2_{\lambda}$ are as introduced at the end of Subsection~\ref{subsec1.1}: they consist of functions $f$ such that $e^{\lambda t}f$ belongs to the corresponding unweighted spaces. By constructing an equivalent problem for \eqref{Co-FUNC-INTRO}, we achieve the stabilization within the framework of classical LQ theory. The core of our quantitative analysis is to establish lower and upper bounds for these two problems. Explicit derivation of the lower bound relies on energy estimates, whereas the upper bound is deduced from quantitative observability inequalities.\\[-2mm]

Let $\lambda>1$. We introduce the following transformation
$$z(t)=e^{\lambda t}y(t), \qquad v(t)=e^{\lambda t}u(t).$$
This reformulates \eqref{EQ-FIRST} as
\begin{equation}\label{TRANS-EQ}\begin{cases}
\partial_tz-\Delta z+az+b\cdot \nabla z-\lambda z=\chi_{\omega}v & \text{in } (0,+\infty)\times \Omega,\\[2mm]
z=0& \text{on }(0,+\infty)\times \partial\Omega,\\[2mm]
z( 0,\cdot)=y_0\in L^2(\Omega).
\end{cases}
\end{equation}
Therefore, in order to achieve the rapid stabilization for the original equation \eqref{EQ-FIRST}, it suffices to establish the exponential stabilization for \eqref{TRANS-EQ}.

We now formulate the optimal control problem
\begin{equation}\label{COFUC-NEW-1}
\Phi_2(y_0):= {\inf}\left\{\mathcal J_2(z,v;y_0)\;\middle|\;
\begin{array}{l}
(z,v)\in W(0,+\infty;H_0^1(\Omega),H^{-1}(\Omega))\times L^2(0,+\infty;L^2(\Omega)),\\[1mm]
(z,v)\ \text{satisfies \eqref{TRANS-EQ}}
\end{array}\right\},
\end{equation}
where
\begin{equation}\label{CO-FUNCAL-1}
\mathcal J_2(z,v;y_0)=\frac12\int_0^\infty\left(\|\nabla z(t)\|_{L^2(\Omega)}^2+\|v(t)\|_{L^2(\Omega)}^2\right)\,dt.
\end{equation} 
It is straightforward to verify that
$$\Phi_2(y_0)=\Phi_1(y_0), \;\;\forall y_0\in L^2(\Omega).$$

We first investigate the upper bound for \eqref{Co-FUNC-INTRO} (equivalently, \eqref{COFUC-NEW-1}). To this end, we construct a feasible control starting from \eqref{OBS-INTRO}. For the parabolic equation \eqref{EQ-FIRST}, this inequality has been proved through a variety of techniques, notably Carleman estimates \cite{FZL,FI}, spectral inequalities \cite{LL,LR}, the frequency function method \cite{WK}, and the propagation of smallness for analytic functions \cite{AEWZ,EMZ}. By the classical duality principle between observability and null controllability, this estimate immediately yields a null control with an explicit cost.

\begin{lemma}\label{Lemma-con-cost-sec2}
Let $T=\dfrac{1}{\sqrt{\lambda}}$. There exists a control $u\in L^2((0,T)\times\Omega)$ such that the corresponding solution to \eqref{EQ-FIRST} satisfies $y(T,\cdot)=0$. Moreover, the associated control cost satisfies
\begin{equation}\label{COst-con-523}
\|u\|_{L^2((0,T)\times\Omega)}\le e^{C(\Omega,d,\omega,a,b)\sqrt{\lambda}}\|y_0\|_{L^2(\Omega)}.
\end{equation}
\end{lemma}

The control constructed in Lemma \ref{Lemma-con-cost-sec2} serves as a feasible control for \eqref{Co-FUNC-INTRO}, and the control cost is crucial for establishing the quantitative upper bound below.

\begin{lemma}\label{PHIBOUND}
There exists a constant $\tilde{C}_1=\tilde{C}_1(\Omega,\omega,d,a,b)>0$ such that
\begin{equation}\label{PHI-bound-upp-523}
{\Phi_1} (y_0) \le e^{\tilde C_1 \sqrt{\lambda}}\|y_0\|_{L^2(\Omega)}^2,\;\;\forall y_0\in L^2(\Omega).
\end{equation}
\end{lemma}

\begin{proof}
We denote by $u$ and $y$, respectively, the control and the corresponding solution to the equation \eqref{EQ-FIRST} obtained in Lemma \ref{Lemma-con-cost-sec2}. Multiplying \eqref{EQ-FIRST} by $2y$, we have
$$\frac{d}{dt}\left(\|y(t)\|_{L^2(\Omega)}^2\right)= -2\|\nabla y(t)\|_{L^2(\Omega)}^2- 2\langle a y(t), y(t)\rangle_{L^2(\Omega)}- 2\langle b \cdot \nabla y(t), y(t)\rangle_{L^2(\Omega)}+ 2\langle \chi_\omega u(t), y(t)\rangle_{L^2(\Omega)}.$$
By Young's inequality, there exists a constant $C_0=C_0(a,b)>0$ such that
$$\frac{d}{dt}\|y(t)\|_{L^2(\Omega)}^2\le-\|\nabla y(t)\|_{L^2(\Omega)}^2+C_0\|y(t)\|_{L^2(\Omega)}^2+\|u(t)\|_{L^2(\Omega)}^2.$$
Integrating over $t \in (0,T)$ and using the control cost estimate \eqref{COst-con-523}, we get 
$$\|y(t)\|_{L^2(\Omega)}^2 + \int_0^t \|\nabla y(s)\|_{L^2(\Omega)}^2 ds\leq \|y_0\|_{L^2(\Omega)}^2+ C_0\int_0^t \|y(s)\|_{L^2(\Omega)}^2 ds+  e^{2C(\Omega,d,\omega,a,b)\sqrt{\lambda}} \|y_0\|_{L^2(\Omega)}^2.$$

By Gr\"onwall's inequality, 
$$\|y(t)\|_{L^2(\Omega)}^2\leq 2 e^{2C(\Omega,d,\omega,a,b)\sqrt{\lambda}} \|y_0\|_{L^2(\Omega)}^2e^{\int_0^{\frac{1}{\sqrt{\lambda}}} C_0 ds}\leq e^{C_1(\Omega,d,\omega,a,b)\sqrt{\lambda}} \|y_0\|_{L^2(\Omega)}^2.$$
Substituting back into the previous inequality, we immediately obtain that
$$\int_0^T \|\nabla y(t)\|_{L^2(\Omega)}^2 dt\leq e^{C_2( \Omega,d,\omega,a,b)\sqrt{\lambda}} \|y_0\|_{L^2(\Omega)}^2.$$
Finally, using the control cost estimate again, we easily deduce
$$\frac{1}{2} \int_0^{\frac{1}{\sqrt{\lambda}}} e^{2\lambda t}\left(\|\nabla y(t)\|_{L^2(\Omega)}^2 + \|u(t)\|_{L^2(\Omega)}^2\right) dt\leq e^{C_3(\Omega,d,\omega,a,b)\sqrt{\lambda}} \|y_0\|_{L^2(\Omega)}^2.$$

Define $\widetilde{u} = \begin{cases} u & \text{on } (0,T], \\ 0 & \text{on } (T,+\infty). \end{cases}$ Then the corresponding solution is $\widetilde{y} = \begin{cases} y & \text{on } (0,T], \\ 0 & \text{on } (T,+\infty). \end{cases}$
Hence,
$$\frac{1}{2} \int_0^\infty e^{2\lambda t}\left(\|\nabla \widetilde{y}(t)\|_{L^2(\Omega)}^2 + \|\widetilde{u}(t)\|_{L^2(\Omega)}^2\right) dt\leq e^{C_3(\Omega,d,\omega,a,b)\sqrt{\lambda}} \|y_0\|_{L^2(\Omega)}^2.$$
Since $y(T)=0$ in $L^2(\Omega)$, the zero extension $\widetilde y$ belongs to
$W(0,+\infty;H_0^1(\Omega),H^{-1}(\Omega))$. It readily follows that $(\widetilde y,\widetilde u)$ is a feasible pair for $\Phi_1(y_0)$. By the definition of $\Phi_1(y_0)$, we complete the proof.
\end{proof}

\begin{remark}\label{obs-quanti-remark}
Instead of directly using the observability inequality, we transform it into a null controllability result with an explicit control cost. By constructing a feasible control, an upper bound for the optimal control problem is derived, and the quantitative form of this upper bound follows from the quantitative control cost. The quantitative observability inequality has been extensively studied in the literature, including classical parabolic equations \eqref{EQ-FIRST}, higher-order parabolic equations \cite{EMZ}, degenerate parabolic equations \cite{LWYZ}, Oseen systems \cite{BT,FCGIP}, and other related models.
\end{remark}

The lower bound can be obtained via a straightforward energy estimate.

\begin{lemma}\label{PHIBOUND-NEW-18}
There exists a constant $\tilde{C}_2=\tilde{C}_2(\Omega,d,a,b)>0$ such that
\begin{equation}\label{PHI-bound-low-523} 
{\Phi_2} (y_0) \ge \tilde C_2\|y_0\|_{L^2(\Omega)}^2,\;\;\forall y_0\in L^2(\Omega).
\end{equation}
\end{lemma}

\begin{proof}
Since $\Phi_2(y_0)=\Phi_1(y_0)<\infty$, the standard minimizing sequence argument combined with the convexity of the cost functional \eqref{CO-FUNCAL-1} implies that $\Phi_2(y_0)$ is attained by a unique optimal pair $(z^*,v^*)$, so that the infimum is actually a minimum. 

Multiplying the equation \eqref{TRANS-EQ} by $2z^*$, we have
\begin{equation}\label{ENE-EST-1}
\begin{aligned}
\frac{d}{dt}\left(\|z^*(t)\|_{L^2(\Omega)}^2\right)&= -2\|\nabla z^*(t)\|_{L^2(\Omega)}^2
- 2\langle a z^*(t), z^*(t)\rangle_{L^2(\Omega)}- 2\langle b \cdot \nabla z^*(t), z^*(t)\rangle_{L^2(\Omega)}\\
&\quad+2\lambda\|z^*(t)\|_{L^2(\Omega)}^2+ 2\langle \chi_\omega v^*(t), z^*(t)\rangle_{L^2(\Omega)}.
\end{aligned}
\end{equation}
From
$$\frac{1}{2} \int_0^\infty  \left(\|\nabla z^*(t)\|_{L^2(\Omega)}^2 + \|v^*(t)\|_{L^2(\Omega)}^2\right) dt<\infty,$$
together with Poincar\'e's inequality, we conclude that
$\liminf\limits_{t\to \infty} \| z^*(t)\|_{L^2(\Omega)} =0$.
Hence, there exists an increasing sequence $\{T_n\}_{n\in\mathbb{N}^+}$ satisfying $T_n > T_{n-1}$ and $\|z^*(T_n)\|_{L^2(\Omega)} \le \frac{1}{n}, \forall\,n\in\mathbb{N}^+$.

Integrating \eqref{ENE-EST-1} over $(0,T_n)$, we obtain
\begin{equation}\label{ENE-EST-2}
\begin{aligned}
\|y_0\|_{L^2(\Omega)}^2 - \|z^*(T_n)\|_{L^2(\Omega)}^2&= 2\int_0^{T_n} \bigg(
\|\nabla z^*(t)\|_{L^2(\Omega)}^2+ \langle a z^*(t), z^*(t)\rangle_{L^2(\Omega)}+ \langle b\cdot\nabla z^*(t), z^*(t)\rangle_{L^2(\Omega)} \\
&\quad - \lambda \|z^*(t)\|_{L^2(\Omega)}^2- \langle \chi_\omega v^*(t), z^*(t)\rangle_{L^2(\Omega)}
\bigg)dt.
\end{aligned}
\end{equation}
Combined with $\|z^*(T_n)\|_{L^2(\Omega)} \le \frac{1}{n}$, this yields
\begin{equation*}
\|y_0\|_{L^2(\Omega)}^2\le \frac{1}{n^2}+ 2\int_0^{T_n} \bigg(\|\nabla z^*(t)\|_{L^2(\Omega)}^2+ \langle a z^*(t), z^*(t)\rangle_{L^2(\Omega)}+ \langle b\cdot\nabla z^*(t), z^*(t)\rangle_{L^2(\Omega)}- \langle \chi_\omega v^*(t), z^*(t)\rangle_{L^2(\Omega)}\bigg)dt.
\end{equation*}
Applying Young's inequality and taking $n\to\infty$, we deduce the lower bound
\begin{equation}\label{LOWER-BOUND}
\tilde C_2(\Omega,d,a,b)\|y_0\|_{L^2(\Omega)}^2
\le \frac12\int_0^\infty \big(
\|\nabla z^*(t)\|_{L^2(\Omega)}^2 + \|v^*(t)\|_{L^2(\Omega)}^2
\big)dt = \Phi_2(y_0).
\end{equation}
This completes the proof.
\end{proof}

\begin{remark}\label{cost-func-remark}
We further compare our cost functional with the classical low-gain functional in the LQ theory, which corresponds to $\alpha=0$ in \eqref{Func-alpha}. In the low-gain case, explicit lower bounds for the associated optimal control problem cannot be established. In contrast, we adopt a high-gain functional with $\alpha=1/2$ in this paper, and explicit lower bounds can be directly deduced by means of energy estimates, as presented in \eqref{LOWER-BOUND}.
\end{remark}

The following lemma is a standard consequence of the LQ theory for parabolic equations. We omit its proof and refer the reader to \cite{B2} for a detailed argument.

\begin{lemma}\label{lemma3.1}
Let $\lambda> 1$. For every $y_0\in L^2(\Omega)$, let $ (z^*,v^*)$ be the optimal pair of the problem \eqref{COFUC-NEW-1}. Then there exists a self-adjoint, positive operator $\mathcal P_\lambda\in \mathcal L(L^2(\Omega))$ such that
\begin{equation}\label{PHI-P-NEW-12}
\Phi_2(y_0)=\frac12\langle \mathcal P_\lambda y_0,y_0\rangle_{L^2(\Omega)},\quad \forall y_0\in L^2(\Omega),
\end{equation}
and the optimal pair satisfies
$$v^*(t)=-\chi_\omega \mathcal P_\lambda z^*(t),\;\; t>0.$$
Moreover, $\mathcal P_\lambda\in\mathcal{L}(L^2(\Omega))$ is the solution of the operator Riccati equation \eqref{PLAMBDA-61-RICCATI}.
\end{lemma}

From \eqref{PHI-bound-upp-523}, the self-adjoint positive operator $\mathcal{P}_{\lambda}$ satisfies
$$\|\mathcal{P}_\lambda\|_{\mathcal L(L^2(\Omega))}=\sup_{\|y\|_{L^2(\Omega)}\le 1}|\langle \mathcal{P}_\lambda y,y\rangle|\le 2e^{\tilde{C}_1\sqrt{\lambda}}.$$ 
Now, we have all the ingredients to prove Theorem \ref{main-th-rapid-stab}.\\[-2mm] 

\noindent\textit{Proof of Theorem \ref{main-th-rapid-stab}.} It can be readily shown that for all $t\ge0$, the time-shifted pair $(z^*(t+\cdot),v^*(t+\cdot))$ constitutes an optimal pair corresponding to the initial state $z^*(t)$. Accordingly,
$$\Phi_2(z^*(t))=\frac12\int_t^\infty\left(\|\nabla z^*(s)\|_{L^2(\Omega)}^2+\|v^*(s)\|_{L^2(\Omega)}^2\right)\,ds,\quad \forall t\ge0.$$

By Lemmas \ref{PHIBOUND} and \ref{PHIBOUND-NEW-18}, it holds that
$$\tilde C_2\|z^*(t)\|_{L^2(\Omega)}^2\le\Phi_2(z^*(t))\le \Phi_2(y_0) {= \Phi_1(y_0)} \le e^{\tilde{C}_1\sqrt{\lambda}}\|y_0\|_{L^2(\Omega)}^2,\quad \forall t\ge0,$$
which implies that
$$\|z^*(t)\|_{L^2(\Omega)}^2\le\tilde C_2^{-1}  e^{\tilde C_1\sqrt{\lambda}}\|y_0\|_{L^2(\Omega)}^2,\quad \forall t\ge0.$$
Considering $y(t)=e^{-\lambda t}z^*(t)$, the previous estimate yields
\begin{equation}\label{y-cost-esti}
\|y(t)\|_{L^2(\Omega)}^2=e^{-2\lambda t}\|z^*(t)\|_{L^2(\Omega)}^2\le\tilde C_2^{-1} e^{\tilde C_1\sqrt{\lambda}}e^{-2\lambda t}\|y_0\|_{L^2(\Omega)}^2,\quad \forall t\ge0.
\end{equation}

Recalling that $v^*=-\chi_\omega \mathcal P_\lambda z^*$, we infer that
$$u(t)=e^{-\lambda t}v^*(t)=-e^{-\lambda t}\chi_\omega\mathcal P_\lambda z^*(t)=-\chi_\omega\mathcal P_\lambda y(t),\quad\forall t>0.$$
In particular, 
\begin{equation}\label{u-cost-esti}
\|u(t)\|_{L^2(\Omega)}\le \|\mathcal{P}_\lambda\|_{\mathcal L(L^2(\Omega))}\|y(t)\|_{L^2(\Omega)}\le 2\tilde C_2^{-\frac 12} e^{\frac 32\tilde C_1\sqrt{\lambda}}e^{-\lambda t}\|y_0\|_{L^2(\Omega)},\quad \forall t>0.
\end{equation}
Hence, \eqref{y-cost-esti} and \eqref{u-cost-esti} yield the desired quantitative rapid stabilization estimate.\hfill$\square$

\section{Equivalence between quantitative rapid stabilization and quantitative observability}\label{EQUV-SEC}

This section is devoted to the proof of Theorem \ref{EQV}. The proof is divided into two parts. First, for systems admitting the quantitative rapid stabilization, we adopt the time iteration strategy proposed in \cite{CN} and follow the parameter selection scheme in \cite{X2}. We construct piecewise feedback laws with increasing decay rates to achieve the null controllability with explicit control costs. The quantitative observability inequality is then established via the duality principle. Conversely, the quantitative LQ framework developed in Section \ref{QRS-INF} gives the quantitative rapid stabilization.\\[-2mm] 

\noindent\textbf{Step 1. We first prove that $e^{\mathcal O(\lambda^{\frac{p}{p+1}})}$-rapidly stabilizable implies $e^{\mathcal O({\frac{1}{T^p}})}$-observable.}

To this end, we introduce the sequence
$$T_0=0,\qquad T_n=T_{n-1}+\frac{T}{2^n},\quad \forall n\in\mathbb N^+,$$
so that $\lim\limits_{n\to+\infty}T_n=T$. We also fix constants
$$Q=\max\{2^{p+2}C,2\}\quad\text{and}\quad M=\frac{Q^{p+1}(2^{p+1}-1)}{2^{p+2}(2^p-1)},$$ 
and define
$$\lambda_n=\frac{Q^{p+1}2^{(p+1)n}}{T^{p+1}},\quad \forall n\in\mathbb N.$$

For each interval $(T_n,T_{n+1}]$, we apply the feedback law associated with the parameter $\lambda_n$, namely the following time-varying feedback control
\begin{equation}\label{c-1}
u(t)=\mathcal U(t,y(t)):= K_{\lambda_n}y(t),\quad \forall t\in (T_n,T_{n+1}].
\end{equation}
Therefore, since the control is given by a linear bounded feedback law on each interval, the closed-loop system admits a unique solution $y \in C([0,T);X)$.

By the quantitative stabilization estimate corresponding to $\lambda_0$, 
we obtain on the first interval $(T_0,T_1]$ that
$$\|y(T_1)\|_{X}\le e^{C\lambda_0^{\frac p{p+1}}-\lambda_0\frac T2}\|y_0\|_{X}\le e^{\frac{Q^{p+1}}{T^p}\frac{1-2^{p+1}}{2^{p+2}}}\|y_0\|_{X}=e^{-\frac{M(2^p-1)}{T^p}}\|y_0\|_{X}.$$
Iterating the same argument on each time interval yields
$$\|y(T_n)\|_{X}\le\prod_{k=0}^{n-1}e^{-2^{kp}\frac{M(2^p-1)}{T^p}}\|y_0\|_{X}=e^{-\frac{M(2^{np}-1)}{T^p}}\|y_0\|_{X},\quad \forall n\in\mathbb N.$$
Moreover, for $t\in (T_n,T_{n+1}]$ with $n\ge 1$, we have
$$\|y(t)\|_{X}\le e^{C\lambda_n^{\frac{p}{p+1}}}e^{-\frac{M(2^{np}-1)}{T^p}}
\|y_0\|_{X}\le e^{-\frac{M(2^{(n-1)p}-1)}{T^p}}\|y_0\|_{X}.$$
Note that $p>0$. Letting $n\to\infty$, and using $\lim\limits_{n\to+\infty}T_n=T$, we conclude that $\lim\limits_{t\to T^-}\|y(t)\|_{X}=0$.

We now estimate the corresponding control cost. For every $n\ge 1$ and $t\in(T_{n},T_{n+1}]$, the quantitative bound on the feedback law gives
$$\|u(t)\|_U\le e^{C\lambda_{n}^{\frac p{p+1}}}\|y(T_{n})\|_{X}\le e^{-\frac{M(2^{(n-1)p}-1)}{T^p}}\|y_0\|_{X}\le e^{\frac{M}{T^p}}\|y_0\|_{X}.$$
Similarly, for $t\in(T_0,T_1]$,
$$\|u(t)\|_U \le e^{C\lambda_0^{\frac p{p+1}}}\|y_0\|_{X}\le e^{\frac{Q^{p+1}}{2^{p+2}T^p}}\|y_0\|_{X}\le e^{\frac{M}{T^p}}\|y_0\|_{X}.$$
Therefore, we obtain that the equation \eqref{ABSTR-EQ-INTR} is null controllable with the following control cost,
$$\|u\|_{L^2(0,+\infty;U)}=\|u\|_{L^2(0,T;U)}\le T^{1/2}\|u\|_{L^\infty(0,+\infty;U)}\le e^{\frac{M}{T^p}}\|y_0\|_{X}.
$$ 
Since $u\in L^2(0,+\infty;U)$, it follows that $y\in C([0,T];X)$. From $\lim\limits_{t\to T^-}\|y(t)\|_{X}=0$, we derive $y(T)=0$ in $X$. By the classical duality principle, this completes the proof that the original system is $e^{\mathcal O({\frac{1}{T^p}})}$-observable.\\[-2mm]

\noindent\textbf{Step 2. We prove the converse implication.}

In what follows, $C_j$ denote generic positive constants independent of $y_0$ and $\lambda$, unless otherwise specified. Arguing as in the proof of Theorem~\ref{main-th-rapid-stab}, we introduce the weighted variables
$$z(t)=e^{\lambda t}y(t),\quad v(t)=e^{\lambda t}u(t).$$
Then the system \eqref{ABSTR-EQ-INTR} is transformed into
\begin{equation}\label{ABSTR-EQ-2}
\begin{cases}
\partial_t z+Az+A_0z-\lambda z=Bv\quad\text{for } t>0,\\[1mm]
z(0)=y_0\in X.
\end{cases}
\end{equation}

We consider the following infinite-horizon LQ optimal control problem
\begin{equation}\label{COFUC-NEW-11}
\Phi_3(y_0):= {\inf}\left\{\mathcal J_3(y,u;y_0)\;\middle|\;
\begin{array}{l}
(y,u)\in W_{\lambda}(0,+\infty;D(A^{\frac{1}{2}}),D(A^{\frac{1}{2}})')\times L^2_{\lambda}(0,+\infty;U),\\[1mm]
(y,u)\ \text{satisfies \eqref{ABSTR-EQ-INTR}}
\end{array}\right\},
\end{equation}
where
\begin{equation}\label{CO-FUNCAL-11}
\mathcal J_3(y,u;y_0)=\frac12\int_0^\infty e^{2\lambda t}\left(\|A^{\frac 12}y(t)\|_{X}^2+
\|u(t)\|_{U}^2\right)\,dt,
\end{equation}
and its equivalent problem
\begin{equation}\label{COFUC-NEW-111}
\Phi_4(y_0):= {\inf}\left\{\mathcal J_4(z,v;y_0)\;\middle|\;
\begin{array}{l}
(z,v)\in W(0,+\infty;D(A^{\frac 12}),D(A^{\frac{1}{2}})')\times L^2(0,+\infty;U),\\[1mm]
(z,v)\ \text{satisfies \eqref{ABSTR-EQ-2}}
\end{array}\right\},
\end{equation}
where
\begin{equation}\label{CO-FUNCAL-111}
\mathcal J_4(z,v;y_0)=\frac12\int_0^\infty\left(\|A^{\frac 12}z(t)\|_{X}^2+\|v(t)\|_{U}^2\right)\,dt.
\end{equation}
It holds that
$$\Phi_4(y_0)=\Phi_3(y_0), \qquad \forall y_0\in X.$$

By the quantitative observability inequality \eqref{q-obs-abstr} and the standard duality argument, taking $T=\lambda^{-\frac1{p+1}}$, we obtain a control $u\in L^2(0,T;U)$ such that the solution $y\in W\big(0,T;D(A^{1/2}),D(A^{1/2})'\big)$ to the equation \eqref{ABSTR-EQ-INTR} satisfies $y(T)=0$. Moreover,
\begin{equation}\label{ABSTR-COST-CONTROL}
\|u\|_{L^2(0,T;U)}\le e^{C\lambda^{\frac{p}{p+1}}}\|y_0\|_X.
\end{equation}

Next, applying the energy estimate to \eqref{ABSTR-EQ-INTR}, we obtain
$$\frac{d}{dt}\|y(t)\|_X^2=-2\langle Ay(t),y(t)\rangle_X-2\langle A_0y(t),y(t)\rangle_X+2\langle Bu(t),y(t)\rangle_X,\;\;\forall t> 0.$$
Since $A_0\in \mathcal L(D(A^{\frac 12}),X)$, by Young's inequality, there exists a constant $C_0=C_0(A_0,B)>0$ such that
\begin{equation}\label{ABSTR-STI-ENERGY}
\frac{d}{dt}\|y(t)\|_X^2\le-\|A^{1/2}y(t)\|_X^2+C_0\|y(t)\|_X^2
+\|u(t)\|_U^2,\;\;\forall t\ge 0.
\end{equation}
Using \eqref{ABSTR-COST-CONTROL} and Gr\"onwall's inequality, we see that
$$\|y(t)\|_X^2\le e^{C_1\lambda^{\frac{p}{p+1}}}\|y_0\|_X^2,\quad \forall t\in[0,T].$$
Substituting this estimate back into the previous inequality \eqref{ABSTR-STI-ENERGY}, we also obtain
\begin{equation}\label{AA1/2-EST}\int_0^T\|A^{1/2}y(t)\|_X^2\,dt\le e^{C_2\lambda^{\frac{p}{p+1}}}\|y_0\|_X^2.
\end{equation}

Define $\widetilde{u} = \begin{cases} u & \text{on } (0,T], \\ 0 & \text{on } (T,+\infty). \end{cases}$ Then the corresponding solution is $\widetilde{y} = \begin{cases} y & \text{on } (0,T], \\ 0 & \text{on } (T,+\infty). \end{cases}$ Since $y(T)=0$ in $X$, the zero extension $\widetilde y$ belongs to $W(0,+\infty;D(A^{1/2}),D(A^{1/2})')$. Hence, $(\widetilde y,\widetilde u)$ is a feasible pair for $\Phi_3(y_0)$. From \eqref{ABSTR-COST-CONTROL} and \eqref{AA1/2-EST}, we have
\begin{equation}\label{Phi-UPPER-bound-abstract}
\Phi_3(y_0)\le e^{ C_3\lambda^{\frac{p}{p+1}}}\|y_0\|_X^2,\quad \forall y_0\in X.
\end{equation}

Since $ \Phi_4(y_0)=\Phi_3(y_0)<\infty$, the continuous embedding of $W\big(0,T;D(A^{1/2}),D(A^{1/2})'\big)$ into $C([0,T];X)$ together with the minimizing sequence analysis ensures that the infimum defining $\Phi_4(y_0)$ is attained by an optimal pair $(z^*,v^*)$, so the infimum is actually a minimum. Uniqueness follows from the convexity of the cost functional \eqref{CO-FUNCAL-111}. Furthermore, by analogous arguments to those for the parabolic case, we obtain the coercive lower bound
\begin{equation}\label{Phi-lower-bound-abstract}
C_4\|y_0\|_X^2\le \Phi_4(y_0),\quad \forall y_0\in X.
\end{equation}
Combining \eqref{Phi-UPPER-bound-abstract} and \eqref{Phi-lower-bound-abstract}, we infer that $\Phi_4$ is a continuous, positive quadratic form on $X$, induced by the symmetric bilinear form $\Upsilon\colon X\times X\to\mathbb{R}$, where
$$\Upsilon(y,w)=\frac14\big(\Phi_4(y+w)-\Phi_4(y-w)\big),\;\;\forall y,w\in X.$$

Hence, by the Riesz representation theorem for continuous quadratic forms on Hilbert spaces, there exists a bounded self-adjoint positive operator $\mathcal P_\lambda\in\mathcal L(X)$ such that
\begin{equation}\label{Phi-representation-abstract}
\Phi_4(y_0)=\frac12\langle\mathcal P_\lambda y_0,y_0\rangle_X,\quad \forall y_0\in X,
\end{equation}
and $\|\mathcal P_\lambda\|_{\mathcal L(X)}=\sup\limits_{\|y\|_X\le 1}|\langle \mathcal P_\lambda y,y\rangle_{X}|\le 2e^{C_3\lambda^{\frac{p}{p+1}}}.$

Given any $T\in (0,+\infty)$, consider the following Bolza problem
\begin{equation}\label{COFUC-Bolza-111}
\Phi_T(y_0):= {\inf}\left\{\mathcal J_T(z,v;y_0)\;\middle|\;
\begin{array}{l}
(z,v)\in W\big(0,T;D(A^{1/2}),D(A^{1/2})'\big)\times L^2(0,T;U),\\[1mm]
(z,v) \text{ satisfies \eqref{ABSTR-EQ-2}},
\end{array}\right\},
\end{equation}
where
\begin{equation}\label{Bolza-FUNCAL-111}
\mathcal J_T(z,v;y_0)=\frac12\int_0^T\left(\|A^{\frac 12}z(t)\|_{X}^2+\|v(t)\|_{U}^2\right)dt+\frac12\langle\mathcal P_\lambda z(T),z(T)\rangle_X.
\end{equation}
According to the dynamic programming principle (cf. \cite[Lemma 3.10]{BR}), the optimal pair $(z^*,v^*)$ of $\Phi_4(y_0)$, when restricted to $(0,T)$, is exactly the optimal pair for $\Phi_T(y_0)$.

By the Karush--Kuhn--Tucker (KKT) theorem for \eqref{COFUC-Bolza-111} (see \cite[Theorem A.1]{BR}), there exists a Lagrange multiplier $(\zeta,q)\in X\times L^2(0,T;D(A^{1/2}))$ such that for any $(\xi,\eta)\in  W\big(0,T;D(A^{1/2}),D(A^{1/2})'\big)\times L^2(0,T;U)$, it holds that
$$\begin{aligned}
&\int_0^T \big(\langle A^{1/2}z^*(t),A^{1/2}\xi(t)\rangle_X+\langle v^*(t),\eta(t)\rangle_U\big)dt
+\langle \mathcal P_\lambda z^*(T),\xi(T)\rangle_X \\
&\quad+\langle \zeta,\xi(0)\rangle_X
+\int_0^T\big\langle \partial_t\xi(t)+A\xi(t)+A_0\xi(t)-\lambda\xi(t)-B\eta(t),\,q(t)\big\rangle_{D(A^{1/2})',D(A^{1/2})}dt=0.
\end{aligned}$$
From the above identity, we deduce that $v^*(t) = B^* q(t)$ for a.e. $t\in (0,T)$, and
$$\begin{cases}
-\partial_t q + Aq + A_0^* q - \lambda q + A z^* = 0\quad \text{in } \;(0,T),\\[4pt]
q(T) = -\mathcal P_\lambda z^*(T).
\end{cases}$$
In particular, since $Az^*\in L^2(0,T;D(A^{\frac 12})')$, we obtain $q\in W(0,T;D(A^{\frac 12}),D(A^{\frac 12})')$ (see \cite{LM}), which in turn yields $q\in C([0,T];X)$.

Fix $s\in(0,T)$. Let $q_s$ be the corresponding KKT multiplier of the Bolza problem $\Phi_s(y_0)$, and set $\widehat q(t)=q(t)-q_s(t)$ for $t\in(0,s)$. Then $v^*(t)=B^*q_s(t)$ for a.e. $t\in(0,s)$, and
$$-\partial_t\widehat q+A\widehat q+A_0^*\widehat q-\lambda\widehat q=0\;\;\text{ in }(0,s),\qquad B^*\widehat q=0\quad\text{a.e. on }(0,s).$$
Indeed, the equation follows by subtraction, while $B^*\widehat q=B^*q-B^*q_s=v^*-v^*=0$ a.e. on $(0,s)$. Let $\phi(t)=e^{\lambda t}\widehat q(t)$. Then $\partial_t\phi-(A+A_0)^*\phi=0$ in $(0,s)$, and $B^*\phi=0$ a.e. on $(0,s)$. 

Let $\tau\in(0,\min\{s,1\})$. Applying the observability inequality
\eqref{q-obs-abstr} on $(0,\tau)$ gives $\phi(0)=0$, and therefore $\widehat q(0)=0$. By uniqueness for the linear evolution equation, $q(t)=q_s(t)$ for $t\in(0,s)$. Using $q_s(s)=-\mathcal P_\lambda z^*(s)$, $q=q_s$ on $(0,s)$, and the continuity of both $q$ and $q_s$, we obtain that
$q(s)=-\mathcal P_\lambda z^*(s)$. Since $s\in(0,T)$ is arbitrary,
$q(t)=-\mathcal P_\lambda z^*(t)$ for $t\in(0,T]$, and thus $v^*(t)=-B^*\mathcal P_\lambda z^*(t)$ for a.e. $t\in(0,T)$. As $T>0$ is arbitrary, the unique optimal control is given by
$v^*=-B^*\mathcal P_\lambda z^*$.

Hence,
$$\Phi_4(z^*(t))=\frac12\int_t^\infty\left(\|A^{\frac 12} z^*(s)\|_{X}^2+\|B^*\mathcal P_\lambda z^*(s)\|_{U}^2\right)\mathrm{d}s=\frac{1}{2}\langle\mathcal{P}_\lambda z^*(t),z^*(t)\rangle_X,\quad \forall\,t\ge0.$$
Fix $z\in D(A)$ and let $z^*$ be the optimal
state starting from $z$. Differentiating the above equality at $t=0$ and using $v^*=-B^*\mathcal P_\lambda z^*$ yield the Riccati equation
$$\langle (-A-A_0+\lambda I)z,\mathcal P_\lambda z\rangle_X-\frac12\|B^*\mathcal P_\lambda z\|_U^2+\frac12\|A^{\frac12}z\|_X^2=0,\quad\forall\,z\in D(A).$$
Moreover, we have
$$C_4\|z^*(t)\|_X^2\le \Phi_4(z^*(t))\le \Phi_4(y_0)=\Phi_3(y_0)\le e^{C_3\lambda^{\frac{p}{p+1}}}\|y_0\|_X^2,\quad \forall t\ge 0.$$

Returning to the original variables, we deduce that
$u=-B^*\mathcal P_\lambda y$. Therefore, the corresponding closed-loop system takes the form
$$\partial_t y+Ay+A_0y+BB^*\mathcal P_\lambda y=0\quad\text{in } \; (0,+\infty),$$
and the associated trajectory satisfies the estimate
$$\|u(t)\|_U+\|y(t)\|_X\le C_4^{-\frac 12}(1+2e^{C_3\lambda^{\frac{p}{p+1}}})e^{\frac 12C_3\lambda^{\frac{p}{p+1}}}e^{-\lambda t}\|y_0\|_X,\quad \forall t>0.$$
The proof is completed.\hfill$\square$

\section{Quantitative rapid stabilization for the Navier--Stokes equations}\label{NS-QRS-SEC}
This section is devoted to the proof of Theorem \ref{main-th-rapid-stab-NS}, which also serves as an application of our quantitative framework to nonlinear equations. By virtue of Theorem \ref{EQV}, we can construct an operator $\mathcal P_\lambda\in\mathcal L(\mathcal H)$ to realize the quantitative rapid stabilization for the linearized equation. We then apply the same feedback law to the nonlinear case. Together with Lyapunov stability analysis, we achieve the local quantitative rapid stabilization of the Navier--Stokes equations around $ y_e$. We also note that $\mathcal{P}_\lambda$ is continuous on $\mathcal{H}_1$, and this property is essential for estimating the nonlinear terms.\\[-2mm] 

We set
$$w:=y-y_e,\quad w_0:=y_0-y_e, \quad \text{and} \quad\pi:=p-p_e. $$
The equation \eqref{NS-EQ-INTRO} can be transformed into 
\begin{equation}\label{NS-LINEARIZED-orginal-NONLINEAR}
\left\{\begin{aligned}
&\partial_tw-\Delta w +(y_e\cdot\nabla)w +(w\cdot\nabla)y_e+(w\cdot\nabla)w+\nabla \pi = \chi_\omega u&& \text{in } (0,+\infty)\times \Omega,\\
&\nabla\cdot w =0&& \text{in } (0,+\infty)\times \Omega,\\
&w=0&& \text{on } (0,+\infty)\times \partial\Omega,\\
&w(0,\cdot)=w_0.
\end{aligned}\right.
\end{equation}

Let $\mathbb P:L^2(\Omega)^2\to \mathcal H$ be the Leray projector. We define the Stokes operator $A:D(A)\subset \mathcal H\to \mathcal H$ by
$$Ay=-\mathbb P\Delta y,\;\;\forall y\in D(A)=H^2(\Omega )^2\cap \mathcal H_1.$$
It is well-known that $A$ is strictly positive, self-adjoint, and satisfies $D(A^{\frac 12}) =\mathcal{H}_1$. Consider the operator $A_0$ with domain $D(A_0)=\mathcal H_1$ defined as
$$A_0 y = \mathbb P\bigl((y_e \cdot \nabla)y + (y \cdot \nabla)y_e\bigr),\quad \forall y\in D(A_0).$$
Note that there exists $c_1>0$ such that
$$|\langle A_0y,w\rangle_{\mathcal{ H}}|\le c_1\|y_e\|_{H^2(\Omega)^2}\|y\|_{\mathcal H_1}\|w\|_{\mathcal{H}},\quad \forall y\in \mathcal H_1 \text{ and } w\in \mathcal H.$$
Next, we introduce the bilinear operator $ a:\mathcal H_1 \times \mathcal H_1\to  \mathcal H_1'$ by
$$\langle a(y,z),w\rangle_{\mathcal H_1',\mathcal H_1}=\int_\Omega (y\cdot\nabla)z\cdot w\,dx,\quad \forall\, y,z,w\in \mathcal H_1.$$
According to \cite[Proposition 2.2]{X1}, there exists $c_2>0$ such that
\begin{equation}\label{NONlinear-est-NS}
\bigl|\langle a(y,z),w\rangle_{\mathcal H_1',\mathcal H_1}\bigr|\le c_2 \|y\|_{\mathcal H}^{\frac{1}{2}}\|y\|_{\mathcal H_1}^{\frac{1}{2}}\|z\|_{\mathcal H}^{\frac{1}{2}}\|z\|_{\mathcal H_1}^{\frac{1}{2}}\|w\|_{\mathcal H_1},\quad \forall\, y,z,w\in \mathcal H_1.
\end{equation}
Finally, denote the control operator $B\in\mathcal L(L^2(\Omega)^2,\mathcal H)$ by $Bu=\mathbb P(\chi_\omega u)$. \\[-2mm]

Based on the above definitions, we can rewrite \eqref{NS-LINEARIZED-orginal-NONLINEAR} as the following evolution equation
\begin{equation}\label{NS-PERTURB-ABS}
\left\{\begin{aligned}
&\partial_t w+Aw+A_0w+a(w,w)=Bu\quad \text{for }t>0,\\
&w(0)=w_0.
\end{aligned}\right.
\end{equation}
To construct a stabilizing feedback law, we first consider the linearized equation
\begin{equation}\label{NS-LINEARIZED-ABS-original}
\left\{\begin{aligned}
&\partial_t w+Aw+A_0w=Bu\quad \text{for }t>0,\\
&w(0)=w_0,
\end{aligned}\right.
\end{equation}
i.e.,
\begin{equation}\label{NS-LINEARIZED-orginal}
\left\{\begin{aligned}
&\partial_tw-\Delta w +(y_e\cdot\nabla)w +(w\cdot\nabla)y_e +\nabla \pi = \chi_\omega u&& \text{in } (0,+\infty)\times \Omega,\\
&\nabla\cdot w =0&& \text{in } (0,+\infty)\times \Omega,\\
&w=0&& \text{on } (0,+\infty)\times \partial\Omega,\\
&w(0,\cdot)=w_0.
\end{aligned}\right.
\end{equation}

The following quantitative observability inequality was established by Buffe and Takahashi in \cite{BT}.
\begin{lemma}\label{NS-OBS-BT-523}
Let $T\in (0,1)$. The observability inequality 
\begin{equation}\label{NS-OBS-523}
\|\varphi(0,\cdot)\|_{\mathcal H} \le e^{C(\Omega,\omega,y_e)/T} \|\varphi\|_{L^2((0,T)\times \omega)^2}
\end{equation}
holds for all $\varphi$ satisfying the adjoint equation
\begin{equation}\label{NS-ADJ}
\left\{\begin{aligned}
&-\varphi_t-\Delta\varphi-\bigl(\nabla\varphi+(\nabla\varphi)^\top\bigr)\,y_e+\nabla\rho=0&&\text{in }(0,T)\times \Omega,\\
&\nabla\cdot\varphi=0&&\text{in } (0,T)\times \Omega,\\
&\varphi=0&&\text{on }(0,T)\times \partial\Omega,\\
&\varphi(T,\cdot)\in \mathcal H.
\end{aligned}\right.
\end{equation}
\end{lemma}
By virtue of the LQ framework developed in Section \ref{EQUV-SEC} (specifically, Theorem \ref{EQV}), we have the following result.

\begin{lemma}\label{lemma3.1-NS}
There exists a self-adjoint, positive operator $\mathcal P_\lambda\in \mathcal L(\mathcal H)$ solving the following Riccati equation
\begin{equation}\label{RI-NS} 
-\langle( A+A_0-\lambda I) w,\mathcal P_\lambda w\rangle_{\mathcal H}+\frac 12\|w\|_{\mathcal H_1}^2-\frac 12	\|B^*\mathcal P_\lambda w\|_U^2=0,
\quad \forall w\in D( A),\end{equation}
which satisfies
\begin{equation}\label{P-RI-NS-BOUND-H-H}
\tilde C_2 \|w\|_{\mathcal H}^2\le\frac{1}{2}\langle \mathcal{P}_\lambda w,w\rangle_{\mathcal{H}}\le e^{\tilde C_1\sqrt{\lambda}}\|w\|_{\mathcal H}^2,\quad \forall w\in\mathcal H,
\end{equation}
where $\tilde C_1=\tilde C_1(\Omega,\omega,y_e)>0$ and $\tilde C_2=\tilde C_2(\Omega,y_e)>0$. Furthermore, there exists a constant $\tilde C_3=\tilde C_3(\Omega,\omega,y_e)>0$ such that $\mathcal P_\lambda\in \mathcal L(\mathcal H_1)$ and
\begin{equation}\label{P-RI-NS-BOUND-V-H}
\|\mathcal P_\lambda\|_{ \mathcal L(\mathcal H_1)}\le e^{\tilde C_3 \sqrt{\lambda}}.
\end{equation}
\end{lemma}

\begin{proof}
It is easy to verify that by taking $X=\mathcal{H}$ and $U=L^2(\Omega)^2$, the operators $A$, $A_0$ and $B$ satisfy all the assumptions in Section \ref{ABS-SUBSET}. Hence, by Theorem \ref{EQV}, it suffices to prove \eqref{P-RI-NS-BOUND-V-H}. In what follows, $C$ and $\hat{C}_j$ denote generic positive constants independent of $w_0$ and $\lambda$, unless otherwise specified. 

We set $T=\frac{1}{\sqrt{\lambda}}$.
To estimate $\mathcal P_\lambda w_0$ in $\mathcal H_1$, we consider the optimality system on the interval $(0,T)$. For any $w_0\in \mathcal{H}_1$, let $(z,v)$ denote the optimal pair for \eqref{COFUC-NEW-111} and $q$ the corresponding adjoint solution, i.e.,
$$\begin{cases}
\partial_t z + Az + A_0 z -\lambda z= Bv& \text{in } (0,T),\\
-\partial_t q + A q + A_0^* q - \lambda q = -A z & \text{in } (0,T),\\
z(0)=w_0\quad \text{and}\quad q(T)=-\mathcal{P}_\lambda z(T).
\end{cases}$$ 
Since $w_0\in\mathcal H_1$, the optimal state $z$ belongs to $L^2(0,T;D(A))$ (cf. \cite{LM}). From $Az\in L^2(0,T;\mathcal{H})$ and $A_0^*\in \mathcal L(\mathcal H_1,\mathcal H)$, we deduce $q\in W(0,T;\mathcal{H}_1,\mathcal{H}_1')$, which further implies $q\in C([0,T];\mathcal{H})$. Moreover, the standard well-posedness theory for parabolic equations (cf. \cite[Chapter 1]{B2})  immediately yields $\sqrt{T-t}\,q\in L^2(0,T;D(A))$ and $\sqrt{T-t}\;\partial_t q\in L^2(0,T;\mathcal{H})$. For any $\delta>0$, the continuous embedding of $W\big(0,T-\delta;D(A),\mathcal{H}\big)$ into $ C([0,T-\delta];\mathcal{H}_1)$ gives $q\in C([0,T);\mathcal{H}_1)$.

By the argument used in the proof of Theorem~\ref{EQV}, applied to the present Oseen optimality system, we have $q(t)=-\mathcal P_\lambda z(t)$ for $ t\in(0,T]$. From $z,q\in C([0,T];\mathcal{H})$ and $\mathcal P_\lambda\in \mathcal L(\mathcal H)$, we obtain that $q(t)=-\mathcal P_\lambda z(t)$ for $ t\in [0,T]$. In particular, since $z(0)=w_0$,
$\mathcal P_\lambda w_0=-q(0)$. We now estimate $\|q(0)\|_{\mathcal H_1}$.

We begin by taking the $\mathcal H$-inner product of the state equation with $2 A z$. This yields
$$\frac{d}{dt}\bigl(\|z(t)\|_{\mathcal H_1}^2\bigr)+\| A z(t)\|_{\mathcal H}^2\le 2\lambda\|z(t)\|_{\mathcal H_1}^2+C\|z(t)\|_{\mathcal H_1}^2+C\|v(t)\|^2_{L^2(\Omega)^2}.$$
Using the bound for the optimal pair 
\begin{equation}\label{P-V-1-NS}
\int_0^\infty \bigl(\|z(t)\|_{\mathcal H_1}^2+\|v(t)\|_{L^2(\Omega)^2}^2\bigr)\,dt\le 2e^{\widetilde C_1\sqrt{\lambda}}\|w_0\|_{\mathcal H}^2,
\end{equation}
together with Gr\"onwall's inequality, we deduce
\begin{equation}\label{P-V-2-NS}
\|z(t)\|_{\mathcal H_1}^2+\int_0^T \,\| A z(t)\|_{\mathcal H}^2\,dt\le e^{\hat C_1\sqrt{\lambda}}\|w_0\|_{\mathcal H_1}^2.
\end{equation}

We now consider the adjoint equation. Define $\widetilde q(t)=q(T-t)$ for $t\in(0,T)$. Then $\widetilde q$ satisfies a forward equation with the nonhomogeneous term $-A z(T-\cdot)$:
\[
\partial_t\widetilde q+A\widetilde q+A_0^*\widetilde q-
\lambda\widetilde q=-Az(T-t)\quad \text{in } (0,T).
\]
Applying the energy estimate technique, we first obtain
$$\int_0^T \|\widetilde q(t)\|_{\mathcal H_1}^2\,dt\le C \Bigl(\|q(T)\|_{\mathcal H}^2+ \int_0^T \|  z(T-t)\|_{\mathcal H_1}^2\,dt\Bigr)e^{CT\lambda}.$$
Using $q(T)=-\mathcal{P}_\lambda z(T)$, \eqref{P-RI-NS-BOUND-H-H}, \eqref{P-V-1-NS} and \eqref{P-V-2-NS}, we get 
\begin{equation}\label{P-V-3-NS}
\int_0^T \|\widetilde q(t)\|_{\mathcal H_1}^2\,dt\le e^{\hat C_2\sqrt{\lambda}}\|w_0\|_{\mathcal H_1}^2.
\end{equation}

Next, taking the $\mathcal H$-inner product of the equation for $\widetilde q$ with $2tA\widetilde q$, we derive
$$\frac{d}{dt}\bigl(t\|\widetilde q(t)\|_{\mathcal H_1}^2\bigr)+t\| A\widetilde q(t)\|_{\mathcal H}^2\le2\lambda t\|\widetilde q(t)\|_{\mathcal H_1}^2+Ct\|\widetilde q(t)\|_{\mathcal H_1}^2		+\|\widetilde q(t)\|_{\mathcal H_1}^2+Ct\|  A z(T-t)\|_{\mathcal H}^2.$$
Integrating over $(0,T)$ and using \eqref{P-V-2-NS}, \eqref{P-V-3-NS}, we arrive at
\begin{equation}\label{P-V-4-NS}
\|\widetilde q(T)\|_{\mathcal H_1}\le T^{-1/2}e^{\hat C_3\sqrt{\lambda}}\|w_0\|_{\mathcal H_1}=\lambda^{1/4}e^{\hat C_3\sqrt{\lambda}}\|w_0\|_{\mathcal H_1}\le e^{\hat C_4\sqrt{\lambda}}\|w_0\|_{\mathcal H_1}.
\end{equation}
Finally, noting that $\widetilde q(T)=q(0)=-\mathcal P_\lambda w_0$, we conclude the proof.
\end{proof}

Now, we are ready to prove Theorem \ref{main-th-rapid-stab-NS}.\\[-2mm]

\noindent\textit{Proof of Theorem \ref{main-th-rapid-stab-NS}.}
We apply the linear bounded feedback law $u=-B^*\mathcal P_\lambda w$. Since $y_0-y_e\in \mathcal H_1$, for any $T\in (0,+\infty)$, the closed-loop system \eqref{NS-PERTURB-ABS} admits a unique strong solution $w\in W\big(0,T;D(A),\mathcal{H}\big)$ (cf. \cite[Theorem 1.7]{B2} and \cite[Theorem 3.10, Chapter III]{Teman}). It follows that $w\in C([0,T];\mathcal{H}_1)$.  

Consider the Lyapunov function
$$\mathcal V(t):=\frac12\langle \mathcal P_\lambda w(t),\,w(t)\rangle_{\mathcal H},\;\; t\ge 0.$$
Differentiating with respect to $t$ yields $\frac{d}{dt}\mathcal V(t)=\langle  w_t(t),\,\mathcal P_\lambda w(t)\rangle_{\mathcal H}$. It readily follows that
$$\frac{d}{dt}\mathcal V=-\langle Aw,\,\mathcal P_\lambda w\rangle_{\mathcal H}
-\langle A_0w,\,\mathcal P_\lambda w\rangle_{\mathcal H}-\|B^*\mathcal P_\lambda w\|_U^2-\langle a(w,w),\,\mathcal P_\lambda w\rangle_{\mathcal H_1',\mathcal H_1}.$$
Together with the Riccati equation \eqref{RI-NS}, we further deduce
$$\frac{d}{dt}\mathcal V(t)+2\lambda \mathcal V(t)\le-\frac 12 \|w(t)\|_{\mathcal H_1}^2+|\langle a(w(t),w(t)),\,\mathcal P_\lambda w(t)\rangle_{\mathcal H_1',\mathcal H_1}|.$$

Using the estimate \eqref{NONlinear-est-NS} for the Navier--Stokes bilinear term, together with \eqref{P-RI-NS-BOUND-V-H}, we infer that
$$\bigl|\langle a(w(t),w(t)),\,\mathcal P_\lambda w(t)\rangle_{\mathcal H_1',\mathcal H_1}\bigr|\le e^{\widetilde C_4\sqrt{\lambda}}\|w(t)\|_{\mathcal H}\|w(t)\|_{\mathcal H_1}^2,$$
for some $\widetilde C_4=\widetilde C_4(\Omega,\omega,y_e)>0$. Hence, we arrive at
$$\frac{d}{dt}\mathcal V(t)+2\lambda \mathcal V(t)\le-\left(\frac12-e^{\widetilde C_4\sqrt{\lambda}}\|w(t)\|_{\mathcal H}\right)\|w(t)\|_{\mathcal H_1}^2.$$

Therefore, if 
\begin{equation}
\label{pri-est-ns} 
\|w(t)\|_{\mathcal H}\le \sigma_\lambda:=\frac {e^{-\widetilde C_4\sqrt{\lambda}}}2,\quad\forall t\ge 0,	
\end{equation}
we have	
$$\frac{d}{dt}\mathcal V(t)+2\lambda \mathcal V(t)\le0,\quad\forall t> 0,$$
which implies that
\begin{equation*}
\tilde{C}_2\|w(t)\|_{\mathcal H}^2\le \mathcal V(t)\le e^{ {-2\lambda t} }  \mathcal V(0)\le e^{\tilde{C}_1\sqrt{\lambda}}e^{{-2\lambda t}} \|w_0\|_{\mathcal H}^2,\quad\forall t\ge 0.
\end{equation*}
Upon modifying the constant, we obtain that for some $\overline{C}_1=\overline{C}_1(\Omega,\omega,y_e)>0$,
\begin{equation}
\label{pri-est-ns-y-new-12} 
\|w(t)\|_{\mathcal H}\le e^{\overline{C}_1\sqrt{\lambda}}e^{{-\lambda t}} \|w_0\|_{\mathcal H},\quad\forall t\ge 0.
\end{equation}
Moreover, from $u=-B^*\mathcal P_\lambda w$, $\|\mathcal P_\lambda\|_{\mathcal L(\mathcal H)}\le 2e^{\tilde C_1\sqrt{\lambda}}$, and \eqref{pri-est-ns-y-new-12}, after enlarging $\overline C_1$ if necessary, we also have $\|u(t)\|_{L^2(\Omega)^2}\le e^{\overline C_1\sqrt{\lambda}}e^{-\lambda t}\|w_0\|_{\mathcal H}$ for all $t>0$.
Therefore, the proof of quantitative rapid stabilization is completed by choosing the initial value $\|y_0-y_e\|_{\mathcal H}$ sufficiently small so that the a priori estimate \eqref{pri-est-ns} holds.  

Indeed, the a priori estimate \eqref{pri-est-ns} holds provided that 
$$\|w_0\|_{\mathcal H}\le \rho_\lambda:={\sigma_\lambda}e^{-\overline{C}_1\sqrt{\lambda}}=\frac{e^{-(\overline{C}_1+\widetilde C_4)\sqrt{\lambda}}}{2}.$$
We prove this assertion. Suppose that there exists $\tilde{t}\in [0,+\infty)$ such that $\|w(\tilde{t})\|_{\mathcal H}\ge \sigma_\lambda$. At $t=0$, we have $ \|w(0)\|_{\mathcal H}< {\sigma_\lambda}$. Thus, since $w\in C([0,+\infty);\mathcal H)$, we define $t^*\in(0,\tilde{t}]$ as the shortest time satisfying $\|w(t^*)\|_{\mathcal H}\ge \sigma_\lambda.$ Furthermore, it holds that $ \|w(t^*)\|_{\mathcal H}={\sigma_\lambda}$, and $\|w(t)\|_{\mathcal{H}}<{\sigma}_\lambda$ for any $t\in [0,t^*)$. 
Hence, $\mathcal V(t)$ is exponentially stable over $[0,{t}^*)$.
This leads to
\begin{equation*}
\begin{aligned}
\|w(t)\|_{\mathcal H}\le e^{\overline{C}_1\sqrt{\lambda}}e^{{-\lambda t}} \|w_0\|_{\mathcal H}\le \sigma_\lambda e^{{-\lambda t}}, \;\forall t\in [0,{t}^*).
\end{aligned}
\end{equation*}
Therefore, we obtain that $\|w(t^*)\|_{\mathcal H}\le \sigma_\lambda e^{{-\lambda t^*}}<\sigma_\lambda$, which contradicts the definition of $t^*$. \hfill$\square$

\section{Finite-dimensional quantitative rapidly stabilizing laws}\label{QRS-FINITE}
This section presents the extension of our quantitative LQ framework to finite-dimensional feedback laws. We first observe that the control in Theorem \ref{main-th-rapid-stab} is in general infinite-dimensional. For self-adjoint systems such as the heat equation, we derive low-frequency quantitative estimates via spectral inequalities and employ the frequency-Lyapunov method for the finite-dimensional quantitative rapid stabilization.

\begin{lemma}\label{notfin}
The range of $\mathcal P_\lambda $ given by Theorem~\ref{main-th-rapid-stab} is not finite-dimensional.
\end{lemma}

\begin{proof} We argue by contradiction. Assume that the range of $\mathcal P_\lambda$ is finite-dimensional. Then $\mathcal P_\lambda$ is a finite-rank operator. Hence, there exists $y\in L^2(\Omega)\setminus \{0\}$ such that $\mathcal P_\lambda y=0$. It readily follows from \eqref{PHI-bound-low-523} and \eqref{PHI-P-NEW-12} that $y=0$. This contradiction shows that the range of $\mathcal P_\lambda $ cannot be finite-dimensional.
\end{proof}

Let $\{(\lambda_n,\varphi_n)\}_{n\ge1}$ be the eigenvalue-eigenfunction pairs of $-\Delta$ in $H^2(\Omega)\cap H^1_0(\Omega)$, with $0<\lambda_1\le \lambda_2\le \cdots$ and $\{\varphi_n\}_{n\ge1}$ forming an orthonormal basis of $L^2(\Omega)$. Let $N(\lambda)$ be the number of eigenvalues $\lambda_n \le \lambda$, counted with multiplicity; equivalently, $\lambda_{N(\lambda)} \le \lambda < \lambda_{N(\lambda)+1}$. We then define the low-frequency projection $\mathcal{E}_\lambda\in\mathcal L(L^2(\Omega))$ by
\begin{equation}\label{low-fre-pro}
\mathcal{E}_\lambda f=\sum_{j=1}^{N(\lambda)} \langle f,\varphi_j\rangle_{L^2(\Omega)}\,\varphi_j,\quad \forall f\in L^2(\Omega).
\end{equation}

Using the frequency-Lyapunov method \cite{X2,XZ}, we establish the finite-dimensional quantitative rapid stabilization for the heat equation. This is accomplished by incorporating refined low- and high-frequency separation into the Lyapunov functional design and balancing their contributions through suitable low-frequency estimates.

\begin{theorem}\label{finite-dim}
There exists a constant $\widehat C=\widehat C(\Omega,d,\omega)>0$ such that, for any $\lambda>1$, one can find  $N=N(\lambda)$ and a positive definite matrix $P_\lambda\in\mathbb R^{N\times N}$ such that for every $y_0\in L^2(\Omega)$, the equation
\begin{equation}\label{closed-loop-low-freq}
\begin{cases}
\partial_ty-\Delta y=\displaystyle\sum_{1\le j\le N}u_j\,\chi_\omega\varphi_j&\text{in }(0,+\infty)\times\Omega ,\\[1mm]
y=0&\text{on }(0,+\infty)\times\partial\Omega ,\\[1mm]
y(0,\cdot)=y_0,
\end{cases}
\end{equation}
with the following feedback law
\begin{equation}\label{feedback-control-low-freq}
U(t):=\begin{pmatrix}u_1(t)\\\vdots\\u_N(t)\end{pmatrix}=-\,B_1^\top P_\lambda\begin{pmatrix}\langle y(t),\varphi_1\rangle_{L^2(\Omega)}\\\vdots\\\langle y(t),\varphi_N\rangle_{L^2(\Omega)}\end{pmatrix},\quad B_1=\bigl(\langle \varphi_i,\varphi_j\rangle_{L^2(\omega)}\bigr)_{1\le i,j\le N},
\end{equation}
admits a unique solution $y\in C\bigl([0,+\infty);L^2(\Omega)\bigr)$ satisfying
\begin{equation}\label{exp-decay-est}
\|U(t)\|_{2}+\|y(t)\|_{L^2(\Omega)}\le e^{\widehat C\sqrt{\lambda}}e^{-\frac\lambda 2 t}\|y_0\|_{L^2(\Omega)},\quad \forall t>0.
\end{equation}
Moreover, $P_\lambda$ is the unique positive definite solution of the algebraic Riccati equation
\begin{equation}\label{ARE-low-freq}
A_1^\top P_\lambda+P_\lambda A_1-P_\lambda B_1B_1^\top P_\lambda+I=0,
\end{equation}
where $A_1=\text{diag }(\frac \lambda 2-\lambda_1,\dots,\frac\lambda 2-\lambda_N)$. 
\end{theorem}

Set $z(t) = e^{\frac {\lambda}2 t} y(t)$, let $v_j(t) = e^{\frac {\lambda}2  t} u_j(t)$ for $j=1,\dots,N$, and denote the control vector by $V(t) = (v_1(t),\dots,v_N(t))^\top$. Then $(z,V)$ satisfies
\begin{equation}\label{TRANS-FIN}
\begin{cases}
\partial_t z - \Delta z - \frac \lambda 2z = \sum_{j=1}^N \chi_\omega v_j \varphi_j & \text{in } (0,+\infty)\times \Omega, \\
z = 0 & \text{on } (0,+\infty)\times \partial\Omega, \\
z(0,\cdot) = y_0 \in L^2(\Omega).
\end{cases}
\end{equation}
Define the low-frequency mode vector of $z(t)$ as $Z(t)=\left( \langle z(t), \varphi_1 \rangle_{L^2(\Omega)}, \dots, \langle z(t), \varphi_N \rangle_{L^2(\Omega)} \right)^\top $.
We then perform a spectral decomposition on \eqref{TRANS-FIN} to obtain the low-frequency system
\begin{equation}\label{ODE-EQ}
\begin{aligned}
\frac{d}{dt}  Z(t)&=A_1Z(t)+B_1V(t).
\end{aligned}
\end{equation}

Let $Y_0 \in\mathbb{R}^N $. We proceed with the stabilization of this low-frequency system first. A classical approach is to formulate the finite-dimensional optimal control problem
\begin{equation}\label{COFUC-NEW-111-finite-ode}
\Phi_5(Y_0):= {\inf}\left\{  \frac{1}{2} \int_0^{\infty} \left( \|Z(t)\|_2^2 + \|V(t)\|_2^2 \right)  dt \;\middle|\;
\begin{array}{l}
(Z,V)\in L^2(0,+\infty;\mathbb R^N)\times L^2(0,+\infty;\mathbb R^N),\\[1mm]
(Z,V)\ \text{satisfies \eqref{ODE-EQ} and $Z(0)=Y_0$} 
\end{array}\right\}.
\end{equation}
The lemma below is used to derive the quantitative upper bound of \eqref{COFUC-NEW-111-finite-ode} and follows directly from spectral inequalities of the Laplace operator; see \cite[Lemma 2.6]{X2} and the relevant spectral inequalities in \cite{LL,LR,LZ}.

\begin{lemma}\label{lemma2.6}
There exists a constant $C=C(\Omega,d,\omega)\ge 1$ such that for any $Y_{N(\lambda)} \in \mathbb{R}^{N(\lambda)}$, it holds that
\begin{equation}\label{spectral-ineq-app}
Y_{N(\lambda)}^\top B_1 Y_{N(\lambda)}\geq e^{-C(\Omega, d,\omega)\sqrt{\lambda}} \|Y_{N(\lambda)}\|^2_2.
\end{equation}
\end{lemma}

For the problem $\Phi_5(Y_0)$, there exists $C_1=C_1(\Omega,d,\omega)>0$ such that
\begin{equation} \label{PHI-ALL-BOUND-2} 
\Phi_5(Y_0)\le e^{C_1\sqrt{\lambda}}\|Y_0\|_2^2,\quad \forall Y_0\in \mathbb R^N.
\end{equation} 
Indeed, we set the control in \eqref{ODE-EQ} as the feedback law $ V(t)=-\lambda e^{C\sqrt{\lambda}}Z(t)$. It readily follows that
$$\frac{d}{dt}\|Z(t)\|_2^2\le Z(t)^\top\Bigl({\lambda}I-2\lambda e^{C\sqrt{\lambda}}B_1\Bigr)Z(t).$$ 
Applying \eqref{spectral-ineq-app}, we obtain $\sigma_{\max}(\lambda I-2\lambda e^{C\sqrt{\lambda}}B_1)\le -{\lambda}$. Hence, $\frac{d}{dt}\|Z(t)\|_2^2\le -{\lambda}\|Z(t)\|_2^2$, which implies that
$$\|Z(t)\|_2^2+\|V(t)\|_2^2\le (1+\lambda^2e^{2C\sqrt{\lambda}})e^{-{\lambda t}}\|Y_0\|_2^2,\quad\forall t\ge 0. $$
Thus \eqref{PHI-ALL-BOUND-2} can be derived directly from the definition of $\Phi_5( Y_0)$.

Thus, as seen in \cite{LZL}, there exists a unique positive definite matrix $P_\lambda$ solving \eqref{ARE-low-freq}, such that
\begin{equation}\label{P-PHI7}
\Phi_5(Y_0)=\frac 12 Y_0^\top P_\lambda Y_0,\quad \forall Y_0\in \mathbb R^N.
\end{equation}
For a symmetric matrix $F$, let $\sigma_{\max}(F)$ and $\sigma_{\min}(F)$ denote its largest and smallest eigenvalues, respectively. The following lemma gives the low-frequency estimates, which are crucial to derive the quantitative result of the original system \eqref{closed-loop-low-freq} within the framework of the frequency-Lyapunov method.

\begin{lemma}\label{P-eigenva}
There exists $\tilde{C}=\tilde{C}(\Omega,d,\omega)>0$ such that
\begin{equation}\label{P-bounded-new}
\sigma_{\max}(P_\lambda)\le e^{\tilde C\sqrt{\lambda}}, \quad \sigma_{\min}(P_\lambda)\ge e^{-\tilde C\sqrt{\lambda}}.
\end{equation}
\end{lemma}

\begin{proof}
On one hand, from \eqref{PHI-ALL-BOUND-2} and \eqref{P-PHI7}, we readily 
deduce that $\sigma_{\max}(P_\lambda)\le 2e^{C_1\sqrt{\lambda}}$. On the other hand, since $P_\lambda$ is the solution of the Riccati equation \eqref{ARE-low-freq}, we have (see \cite{LZL})
$$\sigma_{\min}(P_\lambda)\ge \frac{1}{\sqrt{\sigma_{\min}^2(A_1)+\sigma_{\max}(B_1B_1^\top)}-\sigma_{\min}(A_1)}\ge \frac{1}{\sqrt{9\lambda^2/4+N(\lambda)^2}+3\lambda/2}.$$
This completes the proof by Weyl's law (cf. \cite[Section 3.6]{TW}), which implies that $N(\lambda)\le C(\Omega,d)\lambda^{\frac d2}$.
\end{proof}

We proceed to prove Theorem \ref{finite-dim}.\\[-2mm]

\noindent\textit{Proof of Theorem \ref{finite-dim}.} Let $Y(t) = \left( \langle y(t), \varphi_1 \rangle_{L^2(\Omega)}, \dots, \langle y(t), \varphi_N \rangle_{L^2(\Omega)} \right)^\top $ stand for the low-frequency mode vector of $y(t)$. We adopt the feedback law $U=-B_1^\top P_\lambda Y$. For any $T\in (0,+\infty)$, the closed-loop system admits a unique solution $y\in W(0,T;H^1_0(\Omega),H^{-1}(\Omega))$; we refer to \cite[Theorem 2.3]{X2}. 

Define $\mathcal E_{\lambda}^\perp \in \mathcal{L}(L^2(\Omega))$ by $\mathcal E_{\lambda}^\perp=I-\mathcal E_{\lambda}$. We introduce the Lyapunov function
\begin{equation}\label{Lya}
\begin{aligned}
\mathcal{V}(t)=\gamma_\lambda Y(t)^\top P_\lambda  Y(t) + \|\mathcal{E}_\lambda^\perp y(t)\|_{L^2(\Omega)}^2,\;t\ge0,\end{aligned}
\end{equation} 
where the constant $\gamma_{\lambda}\ge 1$ is to be determined. Clearly, $\mathcal V(t)$ is well-defined. A simple calculation gives
$$\begin{aligned}
\frac{d}{dt}\mathcal V(t)+\lambda\mathcal V(t) &= \gamma_\lambda Y(t)^\top (A_1 P_\lambda + P_\lambda A_1) Y(t) - 2\gamma_\lambda Y(t)^\top (P_\lambda B_1 B_1^\top P_\lambda) Y(t) \\
&\quad+  \left\langle  2\Delta y(t)+\lambda y(t) +2 \sum_{j=1}^N u_j(t) \chi_\omega \varphi_j ,\mathcal{E}_\lambda ^\perp y(t)\right\rangle_{H^{-1}(\Omega),H^1_0(\Omega)}.
\end{aligned}$$

Using the Riccati equation \eqref{ARE-low-freq} satisfied by $P_\lambda$,
we derive
$$\begin{aligned}
\frac{d}{dt}\mathcal V(t)+\lambda\mathcal V(t) &\le -\gamma_\lambda \|Y(t)\|_2^2- 2 \|\nabla \mathcal{E}_\lambda ^\perp y(t)\|_{L^2(\Omega)}^2+\lambda \|\mathcal{E}_\lambda ^\perp y(t)\|_{L^2(\Omega)}^2 + 2\sum_{j=1}^N |u_j(t)| \|\mathcal{E}_\lambda ^\perp y(t)\|_{L^2(\Omega)}.
\end{aligned}$$
From the spectral property $-\|\nabla \mathcal{E}_\lambda ^\perp y(t)\|_{L^2(\Omega)}^2\le -\lambda \|\mathcal{E}_\lambda ^\perp y(t)\|_{L^2(\Omega)}^2 $ and Young's inequality, we further derive
$$\begin{aligned}
\frac{d}{dt}\mathcal V(t)+\lambda\mathcal V(t) \le \bigl( N\sigma_{\max}^2(B_1)\sigma_{\max}^2(P_\lambda)-\gamma_\lambda\bigr)\|Y(t)\|_2^2.\end{aligned}$$

In view of Lemma \ref{P-eigenva}, we take $\gamma_\lambda=e^{C_3\sqrt{\lambda}}$ with a sufficiently large constant $C_3=C_3(\Omega,d,\omega)>0$ such that
$$N\sigma_{\max}^2(B_1)\sigma_{\max}^2(P_\lambda)-\gamma_\lambda\le0,$$
which yields $\frac{d}{dt}\mathcal V(t)\le -\lambda \mathcal V(t)$. Using the eigenvalue bounds of $P_\lambda$, the choice of $\gamma_\lambda$, and $\|U(t)\|_2\le \sigma_{\max}(B_1)\sigma_{\max}(P_\lambda)\|Y(t)\|_2$, we obtain \eqref{exp-decay-est}.\hfill$\square$

\begin{remark}
We recover the finite-dimensional quantitative result by Xiang \cite{X2}. Compared with \cite{X2}, our feedback minimizes the low-frequency energy in \eqref{COFUC-NEW-111-finite-ode} for the transformed equation \eqref{TRANS-FIN}.
\end{remark}

\noindent \textbf{Acknowledgments.} Yu Xiao and Can Zhang were partially supported by NSFC 12422118; Shengquan Xiang was partially supported by NSFC 12571474.

\bibliographystyle{siamplain}  
\bibliography{ref}
\end{document}